\documentclass[ECP]{ejpecp} 

\usepackage{enumerate}  

\SHORTTITLE{Skorokhod's M1 topology for distribution-valued processes} 

\TITLE{Skorokhod's M1 topology for distribution-valued processes} 

\AUTHORS{%
  Sean~Ledger\footnote{Heilbronn Institute, University of Bristol, United Kingdom.
    \EMAIL{sean.ledger@bristol.ac.uk}}
  }

\KEYWORDS{Skorokhod M1 topology ; compacntess and tightness characterisation ; tempered distribution ; countably Hilbertian nuclear space} 

\AMSSUBJ{60G07 ; 60F17} 

\SUBMITTED{December 10, 2015} 
\ACCEPTED{April 11, 2016} 

\ARXIVID{1509.02855} 

\VOLUME{21}
\YEAR{2016}
\PAPERNUM{34}
\DOI{4754}

\ABSTRACT{Skorokhod's M1 topology is defined for c\`adl\`ag paths taking values in the space of tempered distributions (more generally, in the dual of a countably Hilbertian nuclear space). Compactness and tightness characterisations are derived which allow us to study a collection of stochastic processes through their projections on the familiar space of real-valued c\`adl\`ag processes. It is shown how this topological space can be used in analysing the convergence of empirical process approximations to distribution-valued evolution equations with Dirichlet boundary conditions.}

\newcommand{\R}{\ensuremath{\mathbb{R}}}
\newcommand{\E}{\ensuremath{\mathbf{E}}}
\renewcommand{\P}{\ensuremath{\mathbf{P}}}
\newcommand{\N}{\ensuremath{\mathbb{N}}}

\renewcommand{\S}{\ensuremath{\mathcal{S}}}
\newcommand{\Sdual}{\ensuremath{\mathcal{S'}}}

\newcommand{\cadlag}{c\`adl\`ag}
\newcommand{\M}{\ensuremath{\mathrm{M1}}}

\begin{document}




\section{Introduction} 
In \cite{Skorokhod1956} Skorokhod introduced four topologies (labelled J1, J2, M1 and M2) on the space of \cadlag\ paths that take values in a Banach space. The most well-known of these is the J1 topology, which has been used extensively for studying stochastic-process limits \cite{billingsley1999, ethier86, whitt2002}. However, the weaker and less popular M1 topology has an important feature: its modulus of continuity vanishes when a monotone function is passed as an argument \cite[Ch.~12 (4.7)]{whitt2002}. This feature has proven useful in applications such as queuing theory \cite{harrison1996, kellawhitt1990, mandelbaummassey1995, pangwhitt2010, resnick2000, whitt2000, whitt2001}, functional statistics \cite{avramtaqqu1992, basrak2015, duffy2011, krizmani2014, whitt1971, wichura1974}, scaling limits for random walks \cite{benarous2007} and mathematical neuroscience \cite{dirt2015}.

The purpose of this paper is to extend the M1 topology to collections of \cadlag\ processes taking values in the space of tempered distributions or, more generally, in the dual of a countably Hilbertian nuclear space (CHNS). Following the work of It\^o \cite{Ito83}, in which a central limit theorem was developed for distribution-valued processes, the J1 topology was extended to these spaces by Mitoma \cite{mitoma83Tight}. These results were extended by Jakubowski to completely regular range spaces \cite{jakubowki1986} and this is the focus of recent work by Kouritzin \cite{kouritzin2014}. The advantage of working in the dual of a CHNS (as opposed to some Hilbertian subspace, as explored in \cite{pavlyukevich2015}) is that these spaces have a strong finite-dimensional character. Consequently, compactness in the path space can be checked test-function-by-test-function: if $A$ is a collection of \cadlag\ paths taking values in the dual of a CHNS, then $A$ is J1-compact if and only if $\{f(\phi) = (t \mapsto f_{t}(\phi)): f \in A\}$ is J1-compact in the space of real-valued \cadlag\ paths, for every test function, $\phi$, in the CHNS \cite[Thm.~3.1 \&\ 4.1]{mitoma83Tight}. Hence the tightness of a sequence of distribution-valued \cadlag\ processes can be established by projecting down to the familiar space of real-valued \cadlag\ processes. It should be mentioned that contributions like \cite{jakubowki1986}, \cite{blount2010} and \cite{kouritzin2014} also contain results using real-valued projections for more general spaces, but they require establishing point-wise tightness or existence of a limit.

Our aim is to fill the gap in theory between these two settings and to combine the temporal properties of the M1 topology with the spatial properties of the tempered distributions. In Section \ref{Sect_Construction} we construct the M1 topology for \cadlag\ paths taking values in the dual of a CHNS (Definition \ref{Construction_Def_M1}). The corresponding compactness and tightness criteria to those of \cite{mitoma83Tight} are stated and proved in Section \ref{Sect_Characterisation} (Theorems \ref{Char_Thm_Compactness}, \ref{Char_Thm_Tightness} and \ref{Char_Thm_WeakConvergence}). Finally, we use our tools on a concrete example in Section \ref{Sect_Application}, specifically we prove the tightness of a sequence of discrete empirical-measure processes that approximate the solution of a stochastic evolution equation with Dirichlet boundary conditions. Here, the mass lost at the boundary is a monotone process, hence the M1 topology offers a simple decomposition trick for controlling the fluctuations in the approximating sequence (Proposition \ref{Ex_Prop_Decomp}).


\section{Construction of the topology}
\label{Sect_Construction}

We refer the reader to \cite{Becnel_nuclearspace, kallianpur1995, Skorokhod1956, whitt2002} for the basic theory of countably Hilbertian nuclear spaces (CHNS) and the standard Skorokhod topologies for Banach range spaces. Our construction will mirror \cite{mitoma83Tight} for the J1 topology. 

Throughout, $E$ will denote a general CHNS (a specific example being \S, the space of rapidly decreasing functions \cite[Ex.~1.3.2]{kallianpur1995}). The properties we will use are:
\begin{itemize}
\item $E$ is a linear topological space generated by an increasing sequence of Hilbertian semi-norms $\left\Vert \cdot \right\Vert_{0} \leq \left\Vert \cdot \right\Vert_{1} \leq \left\Vert \cdot \right\Vert_{2} \leq \cdots$,

\item The closure, $E_{n}$, of $(E, \left\Vert \cdot \right\Vert_{n})$ is a separable Hilbert space with $E_{0} \supseteq E_{1} \supseteq E_{2} \supseteq \cdots$ and $E = \bigcap_{n \geq 0} E_{n}$,

\item For every $n \geq 0$, there exists $m > n$ such that the canonical injection $E_{m} \hookrightarrow E_{n}$ is Hilbert--Schmidt, that is
\[
\sum_{i=1}^{\infty} \Vert e^{m}_{i} \Vert^{2}_{n} < \infty, 
	\qquad \textrm{whenever } \{ e^{m}_{i} \}_{i \geq 1} \textrm{ is an orthonormal system of } E_{m},
\]

\item The topological duals, $E_{-n} = E'_{n}$, satisfy $\left\Vert \cdot \right\Vert_{0} \geq \left\Vert \cdot \right\Vert_{-1} \geq \left\Vert \cdot \right\Vert_{-2} \geq \cdots$, $E_{0} \subseteq E_{-1} \subseteq E_{-2} \subseteq \cdots$ and $E' = \bigcup_{n \geq 0} E_{-n}$,

\item The \emph{strong topology} of $E'$ is that generated by the collection of semi-norms $\{ p_{B} : E' \to [0,\infty) \}_{B \in \mathcal{B}}$, where $p_{B}(f) := \sup_{x \in B} |f(x)|$ and $\mathcal{B}$ is the collection of bounded subsets of $E$. 
\end{itemize}

\begin{definition}[$D_{E'}$]
\label{Construction_Def_DE'}
The space of \cadlag\ paths, $D_{E'}$, is defined to be the collection of functions mapping $[0,T]$ to $E'$ that are right-continuous and have left limits with respect to the strong topology on $E'$. 
\end{definition}

In practice checking Definition \ref{Construction_Def_DE'} is straightforward due to \cite{mitoma83Cont}. As in the classical case, we define the M1 topology on $D_{E'}$ through a (pseudo-)graph distance on $E' \times [0,T]$. The graph of an element in $D_{E'}$ is formed by joining up its points of discontinuity with intervals (see Figure \ref{Fig_JoinUp}):

\begin{definition}[Interval]
For $f$ and $g$ in $E'$, define the \emph{interval} between these points to be 
\[
\left[f, g\right]_{E'}
  := \left\{ \left( 1 - \lambda\right) f + \lambda g \in E' : \lambda \in \left[0,1\right] \right\}.
\]
If $h_{1}$ and $h_{2}$ are two points in the interval, $h_{1} = (1 - \lambda_{1})f + \lambda_{1} g$ and $h_{2} = (1 - \lambda_{2})f + \lambda_{2} g$, then we write $h_{1} \leq_{[f,g]_{E'}} h_{2}$ whenever $\lambda_{1} \leq \lambda_{2}$. 
\end{definition}

\begin{definition}[Graph]
\label{Construction_Def_Graph}
The \emph{graph} of an element $x \in D_{E'}$ is defined to be the closed subset of $E' \times [0,T]$
\[
\gamma_{x} 
  := \left\{\left(z, t\right) \in E' \times [0,T] : z \in \left[x\left(t-\right), x\left(t\right)\right]_{E'} \right\}.
\]
For $\left(z_{1}, t_{1}\right), \left(z_{2}, t_{2}\right) \in \gamma_{x}$, we say that $\left(z_{1}, t_{1}\right) \leq_{\gamma_{x}} \left(z_{2}, t_{2}\right)$ if either:
\[
\textrm{(i) } t_{1} < t_{2}
	\qquad
	\textrm{ or } 
	\qquad \textrm{(ii) } t_{1} = t_{2} \textrm{ and } z_{1} \leq_{\left[x\left(t_{1}-\right), x\left(t_{1}\right)\right]_{E'}} z_{2}.
\]
\end{definition}

\begin{figure}
\centering
\includegraphics[bb=87bp 155bp 340bp 265bp,clip]{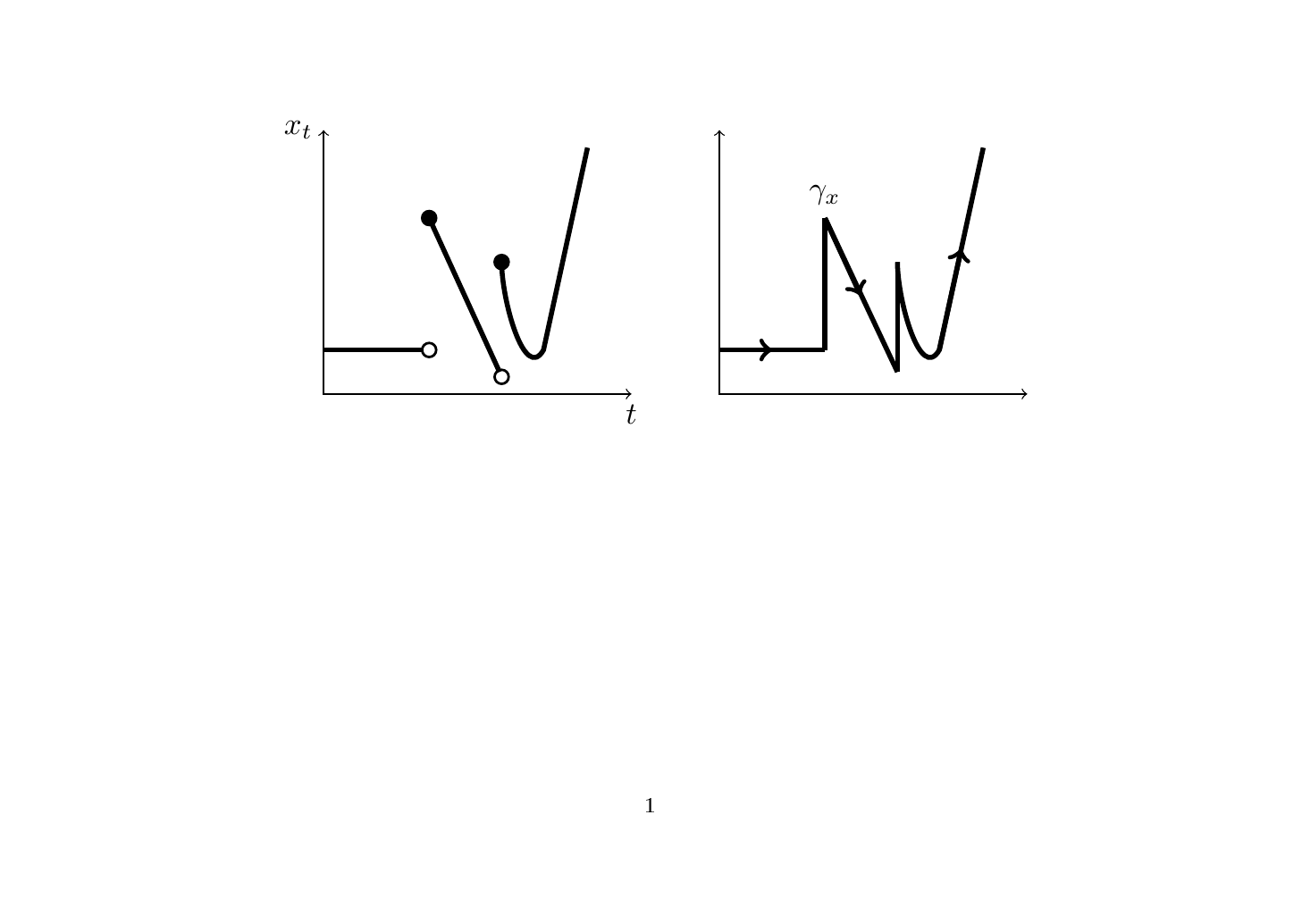}
\vspace{-0.5cm}
\caption{An element of $D_{\R}$ (left) together with its graph, $\gamma_{x} \subseteq \R \times [0,T]$, and graph ordering (right).}
\label{Fig_JoinUp}
\end{figure}

\begin{definition}[Parametric representation]
\label{Construction_Def_ParaRep}
A \emph{parametric representation}, $\lambda_{x}$, of the graph $\gamma_{x}$ is a continuous surjection $\lambda_{x} : [0,1] \to \gamma_{x}$ that is non-decreasing with respect to the graph ordering on $\gamma_{x}$. For $x \in D_{X}$, let the collection of all such parametrisations of $\gamma_{x}$ be denoted $\Lambda_{x}$. 
\end{definition}

We can define a family of pseudometrics on $D_{E'}$ by using the family of semi-norms, $\{p_{B}\}_{B \in \mathcal{B}}$, on $E'$ to measure the graph distance between two parametric representations:

\begin{definition}[A family of pseudometrics]
\label{Construction_Def_Pseudo}
Fix $B \in \mathcal{B}$. Let $x$ and $y$ be elements of $D_{E'}$. For  parametric representations $\lambda_{x} = (z_{x}, t_{x}) \in \Lambda_{x}$ and $\lambda_{y} = (z_{y}, t_{y}) \in \Lambda_{x}$ define 
\[
g_{B}(\lambda_{x}, \lambda_{y})
	:= \sup_{s \in [0,1]} \max \{ p_{B}(z_{x}(s) - z_{y}(s)), |t_{x}(s) - t_{y}(s)| \}.
\]
A pseudometric on $D_{E'}$ is given by 
\[
d_{B, \M}(x, y):= \inf_{\lambda_{x} \in \Lambda_{x},\lambda_{y} \in \Lambda_{y}} g_{B}(\lambda_{x}, \lambda_{y}),
	\qquad \textrm{for } x,y \in D_{E'}.
\]
\end{definition}

The reader can verify that Definition \ref{Construction_Def_Pseudo} gives a pseudometric using \cite[Thm.~12.3.1]{whitt2002}. The family $\{d_{B, \M}\}_{B \in \mathcal{B}}$ generates a topology which we will call the M1 topology on $D_{E'}$:

\begin{definition}[M1 topology]
\label{Construction_Def_M1}
The \emph{M1 topology on $D_{E'}$} is defined to be the projective limit topology of $\{d_{B, \M}\}_{B \in \mathcal{B}}$. That is, the topology with neighbourhoods
\[
\{x \in D_{E'} : d_{B_{i}, \M}(x, x_{0}) < \varepsilon_{i}, \quad \textrm{for } 1 \leq i \leq n\},
\]
where $x_{0} \in D_{E'}$, $B_{i} \in \mathcal{B}$, $\varepsilon_{i} > 0$ and $n \in \N$.
\end{definition}

The following is a collection of facts that are either straightforward or have very similar proofs to the J1 topology case, and so we omit the full details.

\begin{proposition}[Some properties of $(D_{E'}, \M)$]
\label{Con_Prop_Properties}
We have the following:
\begin{enumerate}[(i)]
\item \label{Con_Prop_Properties1} For every $\phi \in E$, the canonical projection
\[
\pi^{\phi} : (D_{E'}, \M) \to (D_{\R}, \M), \qquad \pi^{\phi}(x) = (x_{t}(\phi))_{t \in [0,T]} =: x(\phi)
\]
is continuous,

\item \label{Con_Prop_Properties2} For every $n \geq 0$, the canonical inclusion $\iota_{n} : (D_{E_{-n}}, \M) \to (D_{E'}, \M)$ is continuous,

\item \label{Con_Prop_Properties3} $(D_{E'}, \M)$ is completely regular,

\item \label{Con_Prop_Properties4} Let $\mathfrak{B}(D_{E'})$ denote the Borel $\sigma$-algebra on $(D_{E'},\mathrm{M1})$ and $\mathcal{K}$ the Kolmogorov $\sigma$-algebra generated by the projections
\[
\pi^{\phi_{1}, \phi_{2}, \dots, \phi_{n}}_{t_{1}, t_{2}, \dots, t_{n}} : D_{E'} \rightarrow \R^{n},
  \quad
  \pi^{\phi_{1}, \phi_{2}, \dots, \phi_{n}}_{t_{1}, t_{2}, \dots, t_{n}}(x)
  = (x_{t_{1}}(\phi_{1}), x_{t_{2}}(\phi_{2}), \dots, x_{t_{n}}(\phi_{n})).
\]
Then $\mathfrak{B}(D_{E'}) = \mathcal{K}$. 
\end{enumerate}
\end{proposition}

\begin{proof}
(\ref{Con_Prop_Properties1}) Use that $\{\phi\} \in \mathcal{B}$ for every $\phi \in E$. (\ref{Con_Prop_Properties2}) Straightforward. 

(\ref{Con_Prop_Properties3}) From \cite[Thm.~2.1.1]{kallianpur1995} it suffices to show that for every $x^{1} \neq x^{2} \in D_{E'}$ there exists a bounded set, $B$, such that $d_{B, \M}(x^{1}, x^{2}) > 0$. So take $B = \{ \phi \}$ for any $\phi \in E$ such that $x^{1}(\phi) \neq x^{2}(\phi)$. 

(\ref{Con_Prop_Properties4}) The inclusion $\mathcal{K} \subseteq \mathfrak{B}(D_{E'})$ is straightforward since $\pi^{\phi}_{t} = \pi_{t} \circ \pi^{\phi}$ with the canonical projections
\[
\pi^{\phi} : (D_{E'}, \mathrm{M1}) \rightarrow (D_{\R}, \mathrm{M1})
\quad \textrm{and} \quad
\pi_{t} : (D_{E'}, \R) \rightarrow \R
\]
and $\pi^{\phi}$ is continuous by (\ref{Con_Prop_Properties1}) and $\pi_{t}$ is measurable by \cite[Thm~11.5.2]{whitt2002}. To show $\mathfrak{B}(D_{E'}) \subseteq \mathcal{K}$, let $\iota_{n} : D_{E_{-n}} \hookrightarrow D_{E'}$ be the canonical inclusion. From \cite[Thm.~1.3.1]{kallianpur1995}, $D_{E'} = \bigcup_{n \geq 0} D_{E_{-n}}$, thus for any open subset, $U$, of $(D_{E'}, \mathrm{M1})$
\begin{equation}
\label{eq:MainDef_BigU}
U = \bigcup_{n = 0}^{\infty} \iota_{n}^{-1}(U).
\end{equation}
By (\ref{Con_Prop_Properties2}) $\iota_{n}$ is continuous, so 
\[
\iota_{n}^{-1}(U) 
  \in \mathfrak{B}(D_{E_{-n}}, \mathrm{M1}) 
  = \mathfrak{B}(D_{E_{-n}}, \mathrm{J1}) 
  = \sigma(\pi_{t}:D_{E_{-n}} \rightarrow E_{-n}) 
  \subseteq \sigma(\pi_{t}:D_{E'} \rightarrow E'),
\]
where the first equality follows from \cite[Thm.~11.5.2]{whitt2002} (since $E_{-n}$ is a Banach space) and the second equality follows from the reasoning in \cite[Prop.~3.7.1]{ethier86}.  From \cite[Thm.~3.1.1]{kallianpur1995}, $\mathfrak{B}(E') = \sigma(\pi^{\phi} : E' \rightarrow \R)$, so it follows that $\mathcal{K} = \sigma(\pi_{t} : D_{E'} \rightarrow E')$. Therefore $\iota_{n}^{-1}(U) \in \mathcal{K}$, and the result is complete by (\ref{eq:MainDef_BigU}).
\end{proof}


\section{Compactness and tightness characterisations}
\label{Sect_Characterisation}

The following three theorems characterise compactness, tightness and weak convergence in $(D_{E'}, \M)$. All notation is as defined in Section \ref{Sect_Construction}. 

\begin{theorem}[Compactness characterisation]
\label{Char_Thm_Compactness}
Let $A \subseteq D_{E'}$. The following are equivalent:
\begin{enumerate}[(i)]
\item \label{Char_Thm_Compactness1} $A$ is compact in $(D_{E'}, \mathrm{M1})$,
\item \label{Char_Thm_Compactness2} For every $\phi \in E$, $\pi^{\phi}(A) = \{x(\phi) : x \in A \}$ is compact in $(D_{\R}, \mathrm{M1})$,
\item \label{Char_Thm_Compactness3} There exists $p \in \mathbb{N}$ such that $A \subseteq D_{E_{-p}}$  and $A$ is compact in $(D_{E_{-p}}, \M)$.
\end{enumerate}
\end{theorem}

\begin{theorem}[Tightness characterisation]
\label{Char_Thm_Tightness}
Let $\left(\mu_{n}\right)_{n \geq 1}$ be a sequence of probability measures on $D_{E'}$. Then $\left(\mu_{n}\right)_{n \geq 1}$ is tight on $(D_{E'}, \mathrm{M1})$ if and only if $\left(\mu_{n} \circ (\pi^{\phi})^{-1}\right)_{n \geq 1}$ is tight on $(D_{\R}, \mathrm{M1})$ for every $\phi \in E$. 

Furthermore, if $\left(\mu_{n}\right)_{n \geq 1}$ is tight on $(D_{E'}, \mathrm{M1})$, then it is relatively compact on $(D_{E'}, \mathrm{M1})$.
\end{theorem}

\begin{theorem}[Weak convergence]
\label{Char_Thm_WeakConvergence}
If $\left(\mu_{n}\right)_{n \geq 1}$ is tight on $(D_{E'}, \M)$ and for every $k \geq 1$, $\phi_{1}, \phi_{2}, \dots,\phi_{k} \in E$ and $t_{1}, t_{2}, \dots, t_{k} \in \mathrm{cont}(\mu)$ 
\[
\mu_{n} \circ ( \pi^{\phi_{1}, \phi_{2}, \dots, \phi_{k}}_{t_{1}, t_{2}, \dots, t_{k}} )^{-1}
  \rightarrow
  \mu \circ ( \pi^{\phi_{1}, \phi_{2}, \dots, \phi_{k}}_{t_{1}, t_{2}, \dots, t_{k}} )^{-1},
  \quad \textrm{weakly on } \R^{k},
\]
as $n \rightarrow \infty$, where $\mathrm{cont}(\mu) = \{t \in [0,T] : \mu\{ x \in D_{E'} : x_{t-} = x_{t} \} = 1\}$, then $\left(\mu_{n} \right)_{n \geq 1}$ converges weakly to $\mu$ on $(D_{E'}, \mathrm{M1})$. 
\end{theorem}

Before proving these results we need the following technical result, which exploits the nuclear space structure of $E$ to enable a generic subset $A \subseteq D_{E'}$ to be controlled in terms of its projection onto a finite number of basis vectors. For a normed linear space $X$, we will denote the \emph{M1 modulus of continuity on} $D_{X}$ by
\begin{equation}
\label{Con_eq_Modulus}
w_{X, \M}(x; \delta) := \sup_{t \in [0,T]} 
	\sup_{\mathrm{trip}(t;\delta)} \inf_{\lambda \in [0,1]} \left\Vert x_{t_{2}} - (1 - \lambda)x_{t_{1}} - \lambda x_{t_{3}} \right\Vert_{X},
	\qquad x \in D_{X}, \delta > 0,
\end{equation}
where $\mathrm{trip}(t;\delta) = \{(t_{1}, t_{2}, t_{3}) : \max(0, t - \delta) \leq t_{1} < t_{2} < t_{3} < \min(t+\delta, T)\}$.

\begin{lemma}[Controlling the modulus of continuity]
\label{Technical_Lem_MainLem}
Let $p>n$ be such that the inclusion $E_{p}\hookrightarrow E_{n}$ is Hilbert--Schmidt and let $A\subseteq D_{E'}$ be such that 
\[
c:=\sup_{x\in A}\sup_{t\in\left[0,T\right]}\left\Vert x_{t}\right\Vert _{-n}<\infty.
\]
Then, for every $\varepsilon>0$, there exists $k\geq1$
and $\phi_{1},\phi_{2},\ldots,\phi_{k} \in E$ such that 
\[
\sup_{x \in A} w_{E_{-p}, \mathrm{M1}} (x;\delta) \leq 2c \varepsilon^{-1} \!\! \max_{i=1,\ldots, k} \sup_{x\in A}w_{\R, \mathrm{M1}} (x(\phi_{i});\delta) + 3c\varepsilon,
  \qquad \textrm{for every } \delta >0.
\]
\end{lemma}

The proof is a technical computation for which the next lemma will be helpful. This argument is just a repackaging of the Heine--Borel theorem and is adapted from \cite[Lem.~A.28]{HolleyStroock1979}. 

\begin{lemma}[A geometric argument]
\label{Technical_Lem_Geometric}
Fix $M \in \N$ and $\varepsilon>0$, and let $B$ be the closed unit ball in $\R^{M}$. Then there exists a finite set $\Theta=\left\{ \theta_{1},\ldots,\theta_{k}\right\} \subseteq \R^{M}$ such that $\left\Vert \theta_{i} \right\Vert _{\R^{M}}=1$, and for all $v_{1},v_{2}\in B$ satisfying 
\[
\min_{\lambda \in \left[0,1\right]} \left\Vert \lambda v_{1} + \left( 1-\lambda \right) v_{2} \right\Vert _{\R^{M}} \geq \varepsilon,
\]
there is an $i \in \left\{ 1, 2, \ldots, k \right\}$ for which 
\[
\left\Vert \lambda v_{1} + \left(1-\lambda\right)v_{2} \right\Vert _{\mathbb{R}^{M}}
  \leq 2 \varepsilon^{-1} \left|\left( \lambda v_{1} + \left(1-\lambda\right) v_{2} \right) \cdot \theta_{i} \right|,
\]
for every $\lambda \in \left[0,1\right]$. Here, \textup{\emph{$\left\Vert \cdot\right\Vert _{\mathbb{R}^{M}}$
is the Euclidean norm and $k$ depends on $M$ and $\varepsilon$.}}\textup{ }
\end{lemma}

\begin{proof}
For $v_{1},v_{2}\in B$, let $\lambda_{0}=\lambda_{0}\left(v_{1},v_{2}\right)$ be a minimiser of 
\[
\lambda\mapsto\left\Vert \lambda v_{1}+\left(1-\lambda\right)v_{2}\right\Vert _{\mathbb{R}^{M}}
\]
on $\left[0,1\right]$. Define 
\[
S=\left\{ \left(v_{1},v_{2}\right)\in B\times B:\left\Vert \lambda_{0}v_{1}+\left(1-\lambda_{0}\right)v_{2}\right\Vert _{\mathbb{R}^{M}}\geq\varepsilon\right\},
\]
then for $\left(v_{1},v_{2}\right)\in S$ set $\theta\left(v_{1},v_{2}\right)$
to be the unit vector in the direction of the pair's minimiser: 
\[
\theta\left(v_{1},v_{2}\right)=\frac{\lambda_{0}v_{1}+\left(1-\lambda_{0}\right)v_{2}}{\left\Vert \lambda_{0}v_{1}+\left(1-\lambda_{0}\right)v_{2}\right\Vert _{\mathbb{R}^{M}}},
\]
so that for every $\lambda\in\left[0,1\right]$
\begin{align*}
\left(\lambda v_{1}+\left(1-\lambda\right)v_{2}\right)\cdot\theta\left(v_{1},v_{2}\right)\geq\varepsilon & \geq\varepsilon\left\Vert \lambda v_{1}+\left(1-\lambda\right)v_{2}\right\Vert _{\mathbb{R}^{M}}
 >\frac{\varepsilon}{4}\left\Vert \lambda v_{1}+\left(1-\lambda\right)v_{2}\right\Vert _{\mathbb{\mathbb{R}}^{M}}.
\end{align*}
Hence there is an open neighbourhood, $N\!\left(v_{1},v_{2}\right)$,
in $\mathbb{R}^{M}\times\mathbb{R}^{M}$ about $\left(v_{1},v_{2}\right)$
such that for all $\left(w_{1},w_{2}\right)\in N\!\left(v_{1},v_{2}\right)$
\[
\frac{\varepsilon}{2} \left\Vert \lambda w_{1}+\left(1-\lambda\right)w_{2}\right\Vert _{\mathbb{\mathbb{R}}^{M}}\leq\left(\lambda w_{1}+\left(1-\lambda\right)w_{2}\right)\cdot\theta\left(v_{1},v_{2}\right).
\]
Now $S$ is a compact set with open cover $\left\{ N\!\left(v_{1},v_{2}\right)\right\} _{\left(v_{1},v_{2}\right)\in S}$, therefore it is possible to take a finite sub-cover, $\{ N\!\left(v_{1}^{i},v_{2}^{i}\right)\} _{i=1,\ldots, k}$. To complete the proof let $\Theta=\{ \theta\left(v_{1}^{i},v_{2}^{i}\right)\} _{i=1,\ldots, k}$. 
\end{proof}

\begin{proof}[Proof of Lemma \ref{Technical_Lem_MainLem}]
For readability, introduce the notation
\[
y^{\lambda} 
  = y (\lambda;x;t_{1},t_{2},t_{3}) 
  := x_{t_{2}}-\lambda x_{t_{1}}-(1-\lambda)x_{t_{3}}
  \in E',
\]
where $\lambda \in [0,1]$, $x \in D_{E'}$ and $t_{1},t_{2},t_{3} \in [0,T]$. Then, recalling the notation of (\ref{Con_eq_Modulus}),
\[
w_{E_{-p},\mathrm{M1}} (x;\delta) 
 = \sup_{t,\mathrm{trip}(t;\delta)} \inf_{\lambda \in [0,1]} \bigl\Vert y(\lambda; x; t_{1},t_{2},t_{3}) \bigr\Vert _{-p},
\]
and the triangle inequality gives that $\Vert y^{\lambda} \Vert_{-n}  \leq c$, whenever $x \in A$.

To construct the finite family of vectors, first notice that $E_{-n} \subseteq E_{-p}$, so for an orthonormal basis, $\left\{ e^{p}_{i} \right\}_{i \geq 1}$, of $E_{p}$
\[
\bigl| y^{\lambda}(e^{p}_{i}) \bigr| 
  \leq \bigl\Vert y^{\lambda} \bigr\Vert_{-n} \bigl\Vert e^{p}_{i} \bigr\Vert_{n}
  \leq c\bigl\Vert e^{p}_{i} \bigr\Vert_{n},
\]
for  $x \in A$. Therefore if $M$ is chosen large enough so that $\sum_{i = M + 1}^{\infty} \bigl\Vert e^{p}_{i} \bigr\Vert_{n}^{2} \leq \varepsilon,$ then 
\begin{equation}
\label{eq:Technical_Main1}
\bigl\Vert y^{\lambda} \bigr\Vert_{-p}^{2} 
  \leq \sum_{i = 1}^{M} \bigl| y^{\lambda}(e^{p}_{i}) \bigr|^{2}
  + c \varepsilon,
\end{equation}
so it now suffices to work with $e_{1}^{p}, e_{2}^{p} \dots, e_{M}^{p}$. By introducing the $\R^{M}$ vectors
\[
v_{1} :=
\left(\begin{array}{c}
\left(x_{t_{2}}-x_{t_{1}}\right)\left(e_{1}^{p}\right)\\
\left(x_{t_{2}}-x_{t_{1}}\right)\left(e_{2}^{p}\right)\\
\vdots\\
\left(x_{t_{2}}-x_{t_{1}}\right)\left(e_{M}^{p}\right)
\end{array}\right)
\qquad \textrm{and} \qquad
v_{2}:=\left(\begin{array}{c}
\left(x_{t_{2}}-x_{t_{3}}\right)\left(e_{1}^{p}\right)\\
\left(x_{t_{2}}-x_{t_{3}}\right)\left(e_{2}^{p}\right)\\
\vdots\\
\left(x_{t_{2}}-x_{t_{3}}\right)\left(e_{M}^{p}\right)
\end{array}\right),
\]
the bound in (\ref{eq:Technical_Main1}) gives
\[
w_{E_{-p},\mathrm{M1}} (x;\delta) 
  \leq \sup_{t,\mathrm{trip}(t;\delta)} \inf_{\lambda \in [0,1]} \bigl\Vert \lambda v_{1} + (1-\lambda) v_{2} \bigr\Vert _{\R^{M}}
   + c \varepsilon.
\]

With this fixed choice of $\varepsilon$ and $M$, take $\Theta = \{\theta_{1}, \theta_{2}, \dots, \theta_{k}\}$ from Lemma \ref{Technical_Lem_Geometric}. If it is the case that
\begin{equation}
\label{eq:Technical_Dichotomy}
\inf_{\lambda \in [0,1]} \bigl\Vert \lambda v_{1} + (1-\lambda) v_{2} \bigr\Vert _{\R^{M}} \geq 2c\varepsilon
\end{equation}
then, by normalising $v_{1}$ and $v_{2}$ by $2c$, it follows that there is some $i \in \{1,2,\dots,k\}$ such that
\[
\inf_{\lambda \in [0,1]} \bigl\Vert \lambda v_{1} + (1-\lambda) v_{2} \bigr\Vert _{\R^{M}}
  \leq 2c \varepsilon^{-1} \inf_{\lambda \in [0,1]} \bigl| \theta_{i} \cdot (\lambda v_{1} + (1 - \lambda) v_{2} ) \bigr|.
\]
If (\ref{eq:Technical_Dichotomy}) fails to be true, then clearly the upper bound above can be replaced by $2c\varepsilon$. Hence, for every $x \in A$
\begin{equation}
\label{eq:Technical_Penultimate}
w_{E_{-p},\mathrm{M1}} (x;\delta) 
  \leq 2c\varepsilon^{-1} \!\!\! \sup_{t,\mathrm{trip}(t;\delta)} \max_{i = 1,\dots,k} \inf_{\lambda \in [0,1]} \bigl| \theta_{i} \cdot (\lambda v_{1} + (1 - \lambda) v_{2} ) \bigr|
  + 3c\varepsilon.
\end{equation}

For each $i \in \{1,2,\dots,k\}$, $\phi_{i}$ can now be constructed by defining
\[
\phi_{i} 
  := \sum_{j = 1}^{M} \theta^{(j)}_{i} e_{j}^{p} \in E_{p}, 
\]
where $\theta^{(j)}_{i} \in \R$ is the $j^{\textrm{th}}$ coordinate of $\theta_{i}$. Inequality (\ref{eq:Technical_Penultimate}) therefore reduces to  
\[
w_{E_{-p},\mathrm{M1}} (x;\delta) 
  \leq 2c\varepsilon^{-1} \!\!\!\!\! \sup_{t,\mathrm{trip}(t;\delta)} \max_{i = 1,\dots,k} \inf_{\lambda \in [0,1]} \bigl| y^{\lambda}(\phi_{i}) \bigr|
  + 3c\varepsilon
  \leq 2c\varepsilon^{-1} \!\!\! \max_{i = 1,\dots,k} \! w_{\R, \mathrm{M1}}(x(\phi_{i};\! \delta))
  + 3c\varepsilon
\]
where the second inequality is due to switching the maximum and supremum.  Taking a supremum over $x \in A$ and repeating the switch once more yields the result. 
\end{proof}

By taking $\lambda = 1$ in the above proof, the corresponding result for the increments of $x$ is given:

\begin{corollary}[Controlling increments]
\label{Technical_Cor_IncrementControl}
With $x$, $A$, $c$, $\varepsilon$ and $\phi_{1}, \phi_{2}, \dots, \phi_{k}$ as in the statement of Lemma \ref{Technical_Lem_MainLem}, 
\[
\sup_{x \in A}\sup_{t \in (s-\delta, s+\delta) \cap [0,T]} \left\Vert x_{t} - x_{s}\right\Vert_{E_{-p}}
  \leq 2c\varepsilon^{-1} \max_{i=1, \dots, k} \sup_{x \in A} \sup_{t \in (s-\delta, s+\delta) \cap [0,T]} \left| x_{t}(\phi_{i}) - x_{s}(\phi_{i}) \right| 
  + 3c\varepsilon,
\]
for every $\delta > 0$ and $s \in [0,T]$. 
\end{corollary}

\begin{proof}[Proof of Theorem \ref{Char_Thm_Compactness}]
(\ref{Char_Thm_Compactness1}) $\Rightarrow$ (\ref{Char_Thm_Compactness2}) and (\ref{Char_Thm_Compactness3}) $\Rightarrow$ (\ref{Char_Thm_Compactness1}) follow from Proposition \ref{Con_Prop_Properties} (\ref{Con_Prop_Properties1}) and (\ref{Con_Prop_Properties2}). We now prove (\ref{Char_Thm_Compactness2}) $\Rightarrow$ (\ref{Char_Thm_Compactness3}).

The first half of the proof of \cite[Thm.~2.4.4]{kallianpur1995} does not depend on the choice of temporal topology, hence we have $p > n$ for which $A \subseteq D_{E_{-n}} \subseteq D_{E_{-p}}$, $E_{p} \hookrightarrow E_{n}$ is Hilbert--Schmidt and
\[
c := \sup_{x \in A} \sup_{t \in [0,T]} \left\Vert x_{t} \right\Vert_{-n}
 < \infty.
\]
Therefore $A$ will be compact in $(D_{E_{-p}}, \M)$ if we can verify the second condition in \cite[Thm.~12.12.2]{whitt2002} (with $D_{E_{-p}}$ as the range space).

Let $\phi_{1}, \phi_{2}, \dots, \phi_{k} \in E$ be as in the conclusion of Lemma \ref{Technical_Lem_MainLem}. Since each $\pi^{\phi_{i}}(A)$ is compact (by hypothesis), \cite[Thm.~12.12.2]{whitt2002} implies (recall the notation of (\ref{Con_eq_Modulus}))
\[
\max_{i=1,2,\dots,k} \sup_{x \in A} w_{\R,\M}(x(\phi_{i}); \delta)
	\to 0, \qquad \textrm{as }  \delta \to 0.
\]
Therefore by Lemma \ref{Technical_Lem_MainLem}, $\limsup_{\delta \to 0} \sup_{x \in A} w_{E_{-p}, \M}(x; \delta) \leq 3c\varepsilon$, and likewise for the terms in Corollary \ref{Technical_Cor_IncrementControl}. Since $\varepsilon > 0$ is arbitrary, we are done.
\end{proof}

\begin{proof}[Proof of Theorem \ref{Char_Thm_Tightness}]
To prove the second statement, assume that $(\mu_{n})$ is tight on $(D_{E'}, \M)$. For every $p\geq 0$, $D_{E_{-p}}$ is a Polish space, so \cite[Sec.~3, Def.~2, Ex.~1]{Smolyanov1976} implies that $D_{E_{-p}}$ is a topological Radon space. From \cite[Sec.~3, Ex.~4]{Smolyanov1976} $(D_{E'}, \mathrm{M1})$ is a topological Radon space, hence every probability measure on $(D_{E'}, \mathrm{M1})$ is a Radon measure. By Theorem \ref{Char_Thm_Compactness} (\ref{Char_Thm_Compactness3}), every compact subset of $(D_{E'}, \M)$ is metrizable, therefore \cite[Sec.~5, Thm.~2]{Smolyanov1976} completes the proof of the second result, since $(D_{E'}, \mathrm{M1})$ is completely regular (Proposition \ref{Con_Prop_Properties} (\ref{Con_Prop_Properties3})). 

The first part of the theorem follows from the work in the proof of Theorem \ref{Char_Thm_Compactness} and \cite[Thm.~2.5.1]{kallianpur1995}.
\end{proof}

\begin{proof}[Proof of Theorem \ref{Char_Thm_WeakConvergence}]
Since the Borel and Kolmogorov $\sigma$-algebras on $(D_{E'}, \M)$ coincide (Proposition \ref{Con_Prop_Properties} (\ref{Con_Prop_Properties4})), the result follows by \cite[Prop.~5.1]{mitoma83Tight}.
\end{proof}


\section{Application to empirical processes with Dirichlet boundary conditions}
\label{Sect_Application}

In the remainder of the paper we will show how our machinery can be applied to the problem of approximating stochastic evolution equations through empirical averages of microscopic particles. We will be very concrete and consider a problem from mathematical finance, specifically large portfolio credit modelling. Our analysis will show that $(D_{E'}, \M)$ can be a convenient space on which to prove tightness when studying systems with Dirichlet boundary conditions. 

Define a collection of correlated Brownian motions, $\{X^{i,N}\}_{i=1,\dots, N}$, started on the half-line and evolving with the dynamics
\begin{equation}
X^{i,N}_{t} 
	= X^{i}_{0}
	+ \int^{t}_{0} \rho(L^{N}_{s}) dW_{s}
	+ \int^{t}_{0} \sqrt{1 - \rho(L^{N}_{s})^{2}} dW^{i}_{s},
		\qquad L^{N}_{t} := \frac{1}{N} \sum_{i=1}^{N} \mathbb{1}_{\tau^{i,N} \leq t},
\end{equation}
where $\tau^{i,N} := \inf\{t > 0: X^{i,N}_{t} \leq 0\}$. Here, $W, W^{1}, W^{2}, \dots$ are independent Brownian motions, $\{X^{i}\}_{i \geq 1}$ are i.i.d.\ with some density $f : (0,\infty) \to [0,\infty)$ and $\rho : [0,1] \to [0,1]$ is a measurable function. We will not impose any further regularity constraints on $\rho$ in the following tightness calculations. This is a system in which the proportion of particles that have hit the origin determines the correlation in the system. 

(To see that such $\{X^{i,N}\}_{i=1,\dots,N}$ exist, notice that $t \mapsto L^{N}_{t}$ is piecewise constant. Therefore, to construct the discrete system, take $N$ Brownian motions with initial correlation $\rho(0)$, stop the system at the first hitting of zero, restart the system with correlation $\rho(1/N)$ and repeat.)
 
Our set-up extends the constant correlation model introduced in \cite{bush11}. The motivation for this particular form is to address the correlation skew seen in \cite[Sec.~5]{bush11}. To analyse the model, the quantity of interest is the empirical measure of the population:
\begin{equation}
\nu^{N}_{t} = \frac{1}{N} \sum_{i = 1}^{N} \mathbb{1}_{t < \tau^{i,N}} \delta_{X^{i,N}_{t}}, 
	\qquad \textrm{(hence } L^{N}_{t} = 1 - \nu^{N}_{t}(0,\infty)\textrm{)},
\end{equation}
where $\delta_{x}$ is the usual Dirac delta measure at the point $x \in \R$. 

We would like to establish the weak convergence (at the process level) of $(\nu^{N})_{N \geq 1}$ to some limit $\nu$, which should be the solution of the non-linear evolution equation
\begin{equation}
\label{Ex_Eq_EvoEqu}
d\nu_{t}(\phi) = \frac{1}{2} \nu_{t}(\phi'')dt + \rho(L_{t}) \nu_{t}(\phi')dW_{t},
\qquad L_{t} = 1 - \nu_{t}(0,\infty).
\end{equation}
with test functions $\phi \in \S$ that satisfy $\phi(0) = 0$. This is an example of a (stochastic) McKean--Vlasov equation \cite{sznitman1991}. Proving existence and uniqueness of solutions to this equation would require further regularity constraints on $\rho$. For now, we will only demonstrate that $(\nu^{N})_{N \geq 1}$ is tight on the space $(D_{\Sdual}, \M)$, where \Sdual\ is the space of tempered distributions. 

Notice that, for every $t$, $\nu_{t}^{N}$ is a sub-probability measure, so is an element of \Sdual, and for every $\phi$, $\nu_{t}^{N}(\phi)$ is a real-valued \cadlag\ function. Therefore $\nu^{N}$ has a version that is \cadlag, by \cite{mitoma83Cont}, so $D_{\Sdual}$ is an appropriate space to work with. By Theorem \ref{Char_Thm_Tightness}, it suffices to show $\nu^{N}(\phi)$ is tight in $(D_{\R}, \M)$ for every $\phi \in \S$, and for that it is sufficient to verify \cite[Thm.~12.12.3]{whitt2002}, the first condition of which is trivial since $|\nu^{N}_{t}(\phi)| \leq \Vert \phi\Vert_{\infty}$. For demonstrating that the second condition of \cite[Thm.~12.12.3]{whitt2002} holds, we can employ the helpful result \cite[Thm.~1]{Avram1989}, so to summarise what is now required:

\begin{proposition}
\label{Ex_Prop_Summary}
The sequence $(\nu^{N})_{N \geq 1}$ is tight in $(D_{\Sdual}, \M)$ if, for every fixed $\phi \in \S$, there exists $a, b, c > 0$ such that 
\[
\P(H_{\R}(\nu^{N}_{t_{1}}(\phi), \nu^{N}_{t_{2}}(\phi), \nu^{N}_{t_{3}}(\phi)) \geq \eta)
	\leq c \eta^{-a} |t_{3} - t_{1}|^{1 + b}
\]
for all $N \geq 1$, $\eta > 0$, and $0 \leq t_{1} < t_{2} < t_{3} \leq T$, and 
\[
\lim_{N \to \infty} \P(\sup_{t \in (0,\delta)} |\nu^{N}_{t}(\phi)-\nu^{N}_{0}(\phi)| + \sup_{t \in (T-\delta, T)} |\nu^{N}_{T}(\phi)-\nu^{N}_{t}(\phi)| \geq \eta) = 0,
\qquad \textrm{for every } \eta > 0.
\]
Here, $H_{\R}(v_{1}, v_{2}, v_{3}) := \inf_{\lambda \in [0,1]} |v_{2} - (1-\lambda)v_{1} + \lambda v_{3}|$.
\end{proposition}

The challenge in working with $(\nu^{N})_{N \geq 1}$ is the discontinuity presented by the absorbing boundary at the origin. For the constant $\rho$ case, the authors of \cite{bush11} use explicit estimates based on 2d Brownian motion in a wedge \cite{iyengar1985, metzler2010} to control for the boundary effects. With the more complicated interactions in our present model such methods seem intractable, however, the M1 topology provides an alternative approach.

Introduce the process $\bar{\nu}^{N} \in D_{\Sdual}$ defined by
\[
\bar{\nu}^{N}_{t}
	:= \frac{1}{N} \sum_{i=1}^{N} \delta_{X^{i,N}_{t \wedge \tau^{i,N}}}.
\]
This has the advantage of being continuous (in time), so its increments are easy to control, and it can be related to $\nu^{N}$
through the simple fact
\begin{equation}
\label{Ex_eq_Decom}
\bar{\nu}^{N}_{t}(\phi)
	= \nu^{N}_{t}(\phi) + \phi(0)L^{N}_{t},
	\qquad \textrm{for every } \phi \in \S.
 \end{equation}
Thus, $\nu^{N}$ is a linear combination of a process with well-behaved increments and a process that has zero M1 modulus of continuity ($L^{N}$ is monotone). Substituting (\ref{Ex_eq_Decom}) into $H_{\R}$ from Proposition \ref{Ex_Prop_Summary} gives:

\begin{proposition}[Decomposition trick]
\label{Ex_Prop_Decomp}
For every $\phi \in \S$, $N \geq 1$ and $0 \leq t_{1} < t_{2} < t_{3} \leq T$ 
\[
H_{\R}(\nu^{N}_{t_{1}}(\phi), \nu^{N}_{t_{2}}(\phi), \nu^{N}_{t_{3}}(\phi))
	 \leq |\bar{\nu}^{N}_{t_{1}}(\phi) - \bar{\nu}^{N}_{t_{2}}(\phi)|
	+ |\bar{\nu}^{N}_{t_{2}}(\phi) - \bar{\nu}^{N}_{t_{3}}(\phi)|. 
\]
\end{proposition}

\begin{proof}
Carry out the aforementioned substitution and apply the triangle inequality to get
\[
\textrm{l.h.s.} \leq \textrm{r.h.s.} + |\phi(0)| \inf_{\lambda \in [0,1]} |L^{N}_{t_{2}} - (1 - \lambda) L^{N}_{t_{1}} - \lambda L^{N}_{t_{3}}|.
\]
Since $L^{N}$ is monotone increasing, the final term is zero for $t_{1} < t_{2} < t_{3}$, and this completes the proof.
\end{proof}

Our trick makes the remainder of the tightness proof routine:

\begin{theorem}
\label{Ex_Thm_Tightness}
The sequence $(\nu^{N})_{N \geq 1}$ is tight on $(D_{\Sdual}, \M)$. 
\end{theorem}

\begin{proof}
For $\phi \in E$, let $\Vert\phi\Vert_{\mathrm{lip}}$ denote its Lipschitz constant. Since the $X^{i,N}$ are just 1d Brownian motions, H\"older's inequality gives
\[
\E[|\bar{\nu}^{N}_{t}(\phi)-  \bar{\nu}^{N}_{s}(\phi)|^{4}]
	\leq \frac{1}{N} \sum_{i=1}^{N} \E [|\phi(X^{i,N}_{t \wedge \tau^{i,N}}) - \phi(X^{i,N}_{s \wedge \tau^{i,N}})|^{4}]
	\leq \Vert\phi\Vert_{\mathrm{lip}}^{4} \E[|X^{1,N}_{t \wedge \tau^{1,N}} - X^{1,N}_{s \wedge \tau^{1,N}}|^{4}],
\]
and the final expression is $O(|t-s|^{2})$, uniformly in $N$. Therefore Markov's inequality and Proposition \ref{Ex_Prop_Decomp} give the first statement in Proposition \ref{Ex_Prop_Summary}. 

For the second statement in Proposition \ref{Ex_Prop_Summary}, we can first apply the decomposition in (\ref{Ex_eq_Decom}) to get 
\[
\sup_{t \in (0,\delta)} |\nu^{N}_{t}(\phi) - \nu^{N}_{0}(\phi) |
	\leq \sup_{t \in (0,\delta)} |\bar{\nu}^{N}_{t}(\phi) - \bar{\nu}^{N}_{0}(\phi) | + |\phi(0)|L^{N}_{\delta},
\]
and likewise for the term at $T - \delta$. Using Doob's maximal inequality  and repeating the H\"older calculation above gives
\[
\E\sup_{t \in (0,\delta)} |\nu^{N}_{t}(\phi) - \nu^{N}_{0}(\phi) |
	\leq o(1) + |\phi(0)| \E L^{N}_{\delta} 
	= o(1)+ |\phi(0)| \P(0 < \tau^{1,N} \leq \delta).
\]
Since $X^{1,N}$ is a Brownian motion, the final term is $o(1)$ uniformly in $N$, as $\delta \to 0$. So applying Markov's inequality gives the second statement in Proposition \ref{Ex_Prop_Summary}.
\end{proof}

\begin{remark}[Full convergence]
To prove full weak convergence of $(\nu^{N})_{N \geq 1}$, one approach would be to show that all limit points are supported on solutions of the evolution equation (\ref{Ex_Eq_EvoEqu}) and that this equation has a unique solution. The latter task is dependent on the level of regularity of $\rho$, however the former could be proved using martingale methods. To this end, an application of It\^o's formula gives
\[
d\nu^{N}_{t}(\phi) = \frac{1}{2} \nu^{N}_{t}(\phi'')dt + \rho(L_{t}^{N}) \nu^{N}_{t}(\phi')dW_{t} + dI^{N}_{t}(\phi),
	\qquad \textrm{where} \quad \E |I^{N}_{t}(\phi)|^{2}=O(N^{-1}),
\]
for every $\phi \in \S$ such that $\phi(0) = 0$. By Theorem \ref{Ex_Thm_Tightness}, we have subsequential weak limits for every term in this equation, and the \M\ topology (on \R) is well-behaved with respect to integrals \cite[Thm.~11.5.1]{whitt2002}.
\end{remark}

\begin{remark}[Why not work with $\bar{\nu}^{N}$?]
It might seem easier to work with the process $\bar{\nu}^{N}$ from the start. Notice, however, that the above evolution equation would only hold for $\bar{\nu}^{N}$ if $\phi'(0) = 0 = \phi''(0)$. This has the effect of moving the boundary problems we had in calculating the J1 modulus of continuity of $(\nu^{N})_{N \geq 1}$ to the analysis of the limiting evolution equation.
\end{remark}

\begin{remark}[Why \Sdual?]
\Sdual\ seems an excessively large range space for the above example, however, it is easy to recover that any limiting process must be measure-valued by the Riesz--Markov--Kakutani theorem. Developing a general theory for \Sdual\ also allows us to approach fluctuation problems where the limiting processes would no longer take values in the finite measures. 
\end{remark}





\ACKNO{I am grateful for the helpful suggestions of the anonymous referee. Work supported by the Heilbronn Institute for Mathematical Research and an EPSRC studentship whilst at the University of Oxford.}


\end{document}